\documentclass[a4papersize]{article}

\usepackage{amsmath,amsthm, amsxtra,amssymb,latexsym, amscd}
\usepackage[mathscr]{eucal}
\usepackage{color}
\newtheorem{theorem}{Theorem}[section]
\newtheorem{corollary}[theorem]{Corollary}
\newtheorem{lemma}[theorem]{Lemma}

\newtheorem{example}[theorem]{Example}

\DeclareMathOperator{\depth}{depth}

\DeclareMathOperator{\Ann}{Ann}

\DeclareMathOperator{\Ass}{Ass}
\DeclareMathOperator{\Supp}{Supp}
\DeclareMathOperator{\Var}{Var}
\DeclareMathOperator{\nCM}{nCM}

\DeclareMathOperator{\Att}{Att}
\DeclareMathOperator{\Spec}{Spec}

\begin{document}
\large
\centerline{\Large {\bf  COHEN-MACAULAYNESS OF FORMAL FIBERS }}
\smallskip
\centerline{\Large {\bf  AND DIMENSION OF LOCAL COHOMOLOGY MODULES }}
\smallskip

\medskip
\vskip 0.7cm

\medskip

\vskip 0.7cm
\centerline { TRAN DO MINH CHAU}
\centerline {Thai Nguyen University of Education}
\centerline{Thai Nguyen, Vietnam}
\centerline {E-mail: chautdm@tnue.edu.vn}
\vskip 0.4cm

\medskip
\vskip 0.7cm

\noindent{\bf Abstract} {\footnote{ {\it{Key words and phrases: }} Cohen-Macaulayness; formal fiber; local cohomology module. \hfill\break
  {\it{2000 Subject  Classification: }} 13D45, 13H10.\\ The work is funded by the Ministry of Education and Training of Vietnam under grant number B2026-CTT-08}.  Let $(R, \frak m )$ be a Noetherian local ring, $M$ a finitely generated $R$-module of dimension $d$. Set  $\frak a(M):=\frak a_0(M)\cdots \frak a_{d-1}(M)$, where $\frak a_i(M):=\Ann_RH^i_{\frak m}(M)$ for $i\geq 0$. In this paper, we study the Cohen-Macaulayness of formal fibers of $R$ in the relation with the dimension $\dim (R/\frak a(M)).$ We prove that $\dim (R/\frak a(M))<d$ if and only if $R/\frak p$ is unmixed and the generic formal fiber of $R/\frak p$ is Cohen-Macaulay for all $\frak p\in\Supp_R(M)$ with $\dim (R/\frak p)=d.$ In general,   $R/\frak p$ is unmixed and the generic formal fiber of  $R/\frak p$ is Cohen-Macaulay for all $\frak p\in\Supp_R(M)$ with $\dim (R/\frak p)>\dim (R/\frak a(M)).$  As applications, we explore the structure of local rings and the dimension, the closedness of non Cohen-Macaulay locus of finitely generated modules.

\section{Introduction}
\ \ \ \ Throughout this paper, let $(R,\frak m)$ be a Noetherian local ring, $M$ a finitely generated $R$-module with $\dim_R(M)=d.$ Denote by $\widehat{R}$ and $\widehat{M}$ the $\frak m$-adic completions of $R$ and $M$ respectively.

For $\frak p\in\Spec(R)$, recall that $\widehat{R}\otimes_Rk(\frak p)$ is called the {\it formal fiber ring} of $R$ at $\frak p$ and the spectrum of $\widehat{R}\otimes_Rk(\frak p)$ is called {\it the formal fiber} of $R$ at $\frak p$, where   $k(\frak p):=R_{\frak p}/\frak pR_{\frak p}$ is the residue field.  Note that there is a natural bijection between the formal fiber of $R$ at $\frak p$ and the inverse image of $\frak p$ via the induced map $\Spec(\widehat R)\to \Spec(R).$  If $R$ is a domain then the formal fiber of $R$ at the prime ideal $0$ is called the {\it generic formal fiber} of $R$.  The precise definition of formal fiber  was given  by Grothendieck \cite{G1} in the early 1960s, where  the research on formal fibers continued to develope after the works by  Zariski and Nagata in the 1950s (see \cite{Mat}). 

Many different problems on formal fibers have attracted the interest of mathematicians. About the dimension of formal fibers,  Matsumura \cite{Mat1} proposed to study the number $\alpha(R)$ defined as the maximum of the dimensions of all formal fiber rings of $R$ at prime ideals, he gave examples of local rings $R$ with $\alpha (R)=i$ for each $i\in\{0, \dim(R)-1, \dim(R)-2\}.$ Then Rotthaus \cite{Rott} constructed local rings $R$ with $\alpha (R)=i$ for each $i\in\{1, \ldots, \dim(R)-3\}.$ Regarding the structure of local rings, Heinzer-Rotthaus-Sally \cite{HRS} studied the interrelatedness of properties of the generic formal fiber of a Noetherian local domain $(R,\frak m)$ with the existence of certain Noetherian local domains $C$  birationally dominating $R$ and having $C/\frak mC$ as a finitely generated $R$-module. Charters-Loepp \cite{CL} established a necessary and sufficient condition for the existence of a Noetherian local domain $R$ with a given completion ring and given generic formal fiber. Recently, Z\"{o}schinger and Cuong \cite{Zos, DTC}  investigated local rings $R$ with $\alpha(R)=0$ in order to give characterizations of Weierstrass preparation type and going up property  for the natural homomorphism $R\to \widehat R$. About the Cohen-Macaulayness of formal fibers,  Avramov-Foxby \cite[Theorem 5.3]{AF} proved that if all formal fibers of $R$ are Cohen-Macaulay then so are all formal fibers of $S$ for any local homomorphism $R\to S$ of finite flat dimension. This provided a general solution of  Localization Problem considered by Grothendieck \cite{G1}. 

In this paper, we study the Cohen-Macaulayness of formal fibers in the relation with the dimension of local cohomology modules. Then we apply to study the structure of local rings and the dimension, the closedness of non Cohen-Macaulay locus of finitely generated modules.

From now on, for a finitely generated $R$-module $M$ of dimension $d$ and an integer $i\geq 0$, set $\frak a_i(M):=\Ann_RH^i_{\frak m}(M)$ and
$\frak a(M):=\frak a_0(M)\cdots \frak a_{d-1}(M).$  Following Nagata \cite{Na}, $M$ is said to be {\it unmixed} if $\dim (\widehat R/\frak P)=d$ for all $\frak P\in\Ass_{\widehat R}(\widehat M).$ Note that the definition of the unmixedness used in this paper differs from that in \cite[Page 136]{Mat}. For example, the Noetherian local domain  in Example \ref{E:2a} is unmixed in sense of \cite{Mat}, but not unmixed in sense of \cite{Na}.

The following theorem is the first main result of this paper, that gives some criteria for the Cohen-Macaulayness of formal fibers of $R$ in relation with dimension $\dim (R/\frak a(M))$.

\begin{theorem} \label{T:1a}  The following statements are true.
	\begin{itemize}
		\item [{\rm (a)}] $\dim (R/\frak a(M))<d$ if and only if $R/\frak p$ is unmixed and the formal fiber of  $R$ at $\frak p$ is Cohen-Macaulay for all $\frak p\in\Supp_R(M)$ with $\dim (R/\frak p)=d$.
		\item[{\rm (b)}]  If $\frak p\in\Supp_R(M)$ with $\dim(R/\frak p)>\dim (R/\frak a(M))$ then  $R/\frak p$ is unmixed and the formal fiber of  $R$ at $\frak p$ is Cohen-Macaulay.
	\end{itemize}
\end{theorem}

We denote by $\nCM(M)$ the non Cohen-Macaulay locus of $M$, which is defined as follows
 $$\nCM(M)=\{\frak p\in\Spec(R)|~M_{\frak p}~\text{is not Cohen-Macaulay}\}.$$ 
   In the following cases, $\nCM(M)$ is  a closed subset of $\Spec(R)$ under Zariski topology: 
   $R$ is an excellent ring (see \cite{G1}); $R$ is a quotient of a Gorenstein local ring  (see \cite{Sh1, Sc1,C});
   	 $R$ is a quotient of a Cohen-Macaulay local ring (see \cite{CNN}); all formal fibers of $R$ are Cohen-Macaulay (see \cite{DJ}).   
   	  
  Using Theorem \ref{T:1a}, we examine  the structure of the base ring and the dimension, the closedness of the non Cohen-Macaulay locus $\nCM(M)$. Following Schenzel \cite{Sch} (see also \cite{CN}), the filtration of submodules $H^0_{\frak m}(M)=D_t\subset \ldots \subset D_1\subset D_0=M$ is called {\it the dimension filtration} of $M$ if each $D_i$ is the largest submodule of $M$ of dimension less than $\dim_R(D_{i-1})$. Note that the dimension filtration of $M$ always exists uniquely. The following theorem is the second main result of this paper.
   
  \begin{theorem} \label{T:2a}  Let $H^0_{\frak m}(M)=D_t\subset \ldots \subset D_1\subset D_0=M$ be the dimension filtration of $M$. For each $i\geq 1,$ set ${\cal D}_i:=D_{i-1}/D_i$. The following statements are true.
  	\begin{itemize}
  		\item [{\rm (a)}] If $\dim (R/\frak a({\cal D}_i))\leq 1$ for all $i$ then $R/\Ann_R(M)$ is a quotient of a Cohen-Macaulay local ring, $\nCM(M)$ is closed under Zariski topology and $\dim(\nCM(M))=\dim(\nCM(\widehat M)).$
  		\item[{\rm (b)}]  If $\dim (R/\frak a({\cal D}_i))\leq 2$ for all $i$ then $R/\Ann_R(M)$ is catenary and $\nCM(M)$ is closed under Zariski topology.
  	\end{itemize}
  \end{theorem}  
   
  In general, $\nCM(M)$ is not closed under Zariski topology. We will find in Example \ref{E:2b} a Noetherian local domain $R$ of dimension $3$ such that $\dim (R/\frak a(R))=3$ and $\nCM(R)$ is not closed. Moreover, in Examples \ref{E:1g}, \ref{E:2a}, there is a Noetherian local domain $R$ such that $\dim (R/\frak a(R))=2,$  $\nCM(R)$ is closed but $\dim\nCM(R)<\dim\nCM(\widehat R).$ Therefore, the results of Theorem \ref{T:2a} are optimal in some sense.
  
   In Section 2, we give some preliminaries that will be used later. Section 3 is devoted to the proofs of Theorems \ref{T:1a}, \ref{T:2a}. 

\section{Preliminaries}

\ \ \ \ \ From now on, for a finitely generated $R$-module $M$ of dimension $d$ and each integer $i\leq d$ we  set $\frak a_i(M):=\Ann_RH^i_{\frak m}(M)$. Set
$$\frak a(M):=\frak a_0(M)\frak a_1(M)\cdots \frak a_{d-1}(M).$$
The following splitting result for local cohomology modules (see \cite[Lemma 3.2]{NQ}) is very useful for the proof of the main result of this paper. 

\begin{lemma}  \label{L:1d} Let $U_M(0)$ be the largest submodule of $M$ of dimension strictly less than $d$. Set $\overline M=M/U_M(0)$.  If  $x\in\frak a(M)^3$  is a parameter element of $M$ then for all $i<d-1$  we have 
	$$H^i_{\frak m}(M/xM)\cong H^i_{\frak m}(M)\oplus H^{i+1}_{\frak m}(\overline{M}).$$
\end{lemma}

It is clear that if $R$ is a quotient of a Cohen-Macaulay local ring then $R$ is universally catenary and all formal fibers of $R$ are Cohen-Macaulay. The following important result was proved by Kawasaki \cite[Corollary 1.2]{Kaw}} which shows that the converse statement is  true.

\begin{lemma}\label{L:1c} $R$ is a quotient of a Cohen-Macaulay local ring if and only if $R$ is universally catenary and all formal fibers of $R$ are Cohen-Macaulay.
\end{lemma}

The set of attached primes for Artinian modules introduced  by  Macdonald \cite{Mac}  makes an important role similar to the set of associated primes for Noetherian modules.  Let $A$ be an Artinian $R$-module. Then $A$ has a minimal secondary representation $A=A_1+\ldots +A_n,$ where $A_i$ is $\frak p_i$-secondary. The set $\{\frak p_1,\ldots ,\frak p_n\}$ is independent of the choice of the minimal  secondary representation of $A$. This set is called {\it the set of attached primes} of $A$ and is denoted by  $\Att_RA$.     Here are some basic properties of attached primes of Artinian modules, see \cite{Mac}, \cite[8.2.4, 8.2.5]{BS}. For an ideal $I$ of $R$, denote by $\Var (I)$ the set of all prime ideals of $R$ containing $I$.

\begin{lemma}  \label{L:1a} Let $A$ be an Artinian $R$-module. Then
\begin{itemize}
\item[\rm{(a)}] $\min\Att_RA=\min\Var (\Ann_RA).$ In particular,  $A\neq 0$ if and only if $\Att_RA\neq \emptyset$.
\item[\rm{(b)}]  $A$  has a natural structure as an Artinian $\widehat R$-module and $$\Att_RA=\{\frak P\cap R\mid \frak P\in\Att_{\widehat R}A\}.$$
\end{itemize}
\end{lemma}

Note that   $H^i_{\frak m}(M)$ is an Artinian $R$-module for all  $i\geq 0$. Here are some properties of attached primes of $H^i_{\frak m}(M)$, see \cite[11.3.9]{BS},  \cite[Corollary 4.2]{CNN}.  

\begin{lemma}   \label{L:1b} The following statements are true.
\begin{itemize}
\item[\rm{(a)}] If $\frak p\in\Ass_R(M)$ and $\dim (R/\frak p)=i$, then $\frak p\in\Att_RH^i_{\frak m}(M).$
\item[\rm{(b)}]  Suppose that $R$ is a quotient of a Cohen-Macaulay ring and $\frak p\in\Att_RH^i_{\frak m}(M)$. Then $\dim (R/\frak p)\leq i$. If $\dim (R/\frak p)=i$ then $\frak p\in\Ass_R(M)$.
\end{itemize}
\end{lemma}

We recall the formulas of depth and dimension under flat local homomorphism, see \cite[Theorem A11, Proposition 1.2.16]{BH}.

\begin{lemma} \label{L:bs} Let $R\to S$  be a flat local homomorphism between Noetherian local rings. Then
	\begin{align} \dim_S(M\otimes_RS)&=\dim_R(M)+\dim (S/\frak mS);\notag\\
		\depth_S(M\otimes_RS)&=\depth_R(M)+\depth (S/\frak mS).\notag
		\end{align}
		\end{lemma}

As mentioned in the introduction, $\nCM (M)$ is not a closed subset of  $\Spec(R)$ under Zariski topology in general (cf. Example \ref{E:2b}). Note that $\nCM (M)$ is always closed under specialization, i.e. for any prime ideals $\frak q\subseteq \frak p$ of $R$, if $\frak q\in\nCM(M)$ then $\frak p\in\nCM (M)$. So, we can define its dimension as follows
$$\dim_R(\nCM(M))=\max\{\dim(R/\frak p)\mid \frak p\in\nCM (M)\}.$$

\begin{lemma} \label{L:1e} The following statements are true. 
\begin{itemize}
\item[\rm{(a)}] $\nCM(M)\subseteq \{\frak P\cap R\mid \frak P\in\nCM(\widehat M)\}.$ In particular, $\dim_R\nCM(M)\leq \dim_{\widehat R}\nCM(\widehat M).$
\item[{\rm(b)}]  $\nCM(M)$ is closed under Zariski topology if and only if it has only finitely many minimal elements (under the inclusion).
\item[{\rm(c)}] $\nCM(\widehat{M})$ is  closed under Zariski topology and $\min\nCM(\widehat{M})$ is a finite set.
\end{itemize}
\end{lemma} 
\begin{proof} (a) Let $\frak p\in\nCM(M).$ We choose $\frak P\in\Ass(\widehat R/\frak p\widehat R)$ such that $\dim (\widehat R/\frak P)=\dim (R/\frak p).$  Then $\frak P\cap R=\frak p$ and $\widehat R_{\frak P}/\frak p\widehat R_{\frak P}$ is Cohen-Macaulay of dimension $0$.  Since the natural map $R_{\frak p}\rightarrow \widehat R_{\frak P}$ is faithfully flat and $\widehat M_{\frak P}\cong M_{\frak p}\otimes_{R_{\frak p}}\widehat R_{\frak P},$ it follows by Lemma \ref{L:bs} that $\frak P\in\nCM (\widehat M).$ Thus, $\nCM(M)\subseteq \{\frak P\cap R\mid \frak P\in\nCM(\widehat M)\}$ and the rest statement follows.

The statement (b) follows by the fact that $\nCM (M)$ is closed under specialization. The statement (c) is derived by the assertion (b). 
\end{proof}

The following lemma gives a criterion for the closedness of the non-Cohen-Macaulay locus in term of the Cohen-Macaulayness of finitely many formal fibers of $R$.

\begin{lemma} \label{P:1e}  The following statements are equivalent:

\begin{itemize}
\item[\rm{(a)}] $\nCM(M)=\{\frak P\cap R\mid \frak P\in\nCM(\widehat M)\}$.
\item[\rm{(b)}]  $\widehat R_{\frak P}/\frak p\widehat R_{\frak P}$ is Cohen-Macaulay for all  $\frak P\in\min\nCM(\widehat M)$ and $\frak p=\frak P\cap R.$ 
\end{itemize}
 If the conditions (a), (b) satisfy then $\dim_R\nCM(M)=\dim_{\widehat R}\nCM(\widehat M)$ and   $$\nCM(M)=\bigcup_{\frak P\in \min\nCM(\widehat M)}\Var(\frak P\cap R)$$ which is  closed under Zariski topology. 
 \end{lemma}

\begin{proof} Set $X=\{\frak P\cap R\mid \frak P\in\min\nCM(\widehat M)\}.$ By Lemma \ref{L:1e}(c), $X$ is a finite set.

(a)$\Rightarrow$(b). Let  $\frak P\in\min\nCM(\widehat M)$ and  $\frak p=\frak P\cap R.$ Then we get by assumption (a) that $\frak p\in\nCM(M)$. Let $\frak P'\in\min\Var (\frak p\widehat R)$ such that $\frak P'\subseteq \frak P$. Then $\frak P'\cap R=\frak p$. As $\frak p\in\nCM(M)$, it follows by Lemma \ref {L:bs} that $\frak P'\in\nCM(\widehat M).$ Hence $\frak P=\frak P'$ by the minimality of $\frak P$. Hence  $\frak P\in \min\Var (\frak p\widehat R).$ So, $\widehat R_{\frak P}/\frak p\widehat R_{\frak P}$ is Cohen-Macaulay. 

(b)$\Rightarrow$(a).  Let $\frak P\in\nCM(\widehat M)$. Set $\frak p:=\frak P \cap R.$ Then $\frak P\supseteq \frak Q$ for some prime ideal $\frak Q\in\min\nCM (\widehat M).$ Set $\frak q:=\frak Q\cap R.$ Then $\frak p\supseteq \frak q$  and $\frak q\in X$.  As $\widehat R_{\frak Q}/\frak q\widehat R_{\frak Q}$ is Cohen-Macaulay by assumption (b) and $\frak Q\in\nCM(\widehat M)$, it follows by Lemma \ref{L:bs} that $\frak q\in\nCM (M).$ Since $\nCM(M)$ is closed under specialization, $\frak p\in\nCM (M).$ So, by Lemma \ref{L:1e}(a),
$$\nCM(M)= \{\frak P\cap R\mid \frak P\in\nCM(\widehat M)\}.$$ 

Assume that one of the equivalent conditions (a), (b) satisfies. It follows by (a) that $\nCM(M)=\bigcup_{\frak p\in X}\Var(\frak p)$.  As $X$ is a finite set, $\nCM(M)$ is a closed. Moreover, it follows by condition (a) that $\dim_R\nCM(M)= \dim_{\widehat R}\nCM(\widehat M)$. 
\end{proof} 

The following example gives a Noetherian local domain $R$ such that $\nCM(R)$ is closed, but  $\dim (\nCM(R))<\dim (\nCM(\widehat R))$ and  $\widehat R_{\frak P}/\frak p\widehat R_{\frak P}$ is not Cohen-Macaulay for some prime ideal  $\frak P\in\min \nCM(\widehat R)$ and  $\frak p=\frak P\cap R$. 

\begin{example} \label{E:1g} {\rm Consider the Noetherian local domain $(R,\frak m)$ of dimension $2$ constructed by Ferrand and Raynaud \cite{FR} such that $\widehat R$ has an embedded  prime $\frak Q$ of dimension $1$. It is clear that $\nCM(R)=\{\frak m\}$, $\frak Q\in\min\nCM(\widehat R)$ and $0=\frak Q\cap R\notin\nCM (R).$ Thus, $\dim (\nCM(R))=0,$ $\dim (\nCM(\widehat R))=1$ and $$\nCM(R)\neq \{\frak P\cap R\mid \frak P\in\nCM(\widehat R)\}.$$}
\end{example}

\section{Main results}

\ \ \ \   We prove the first main result of this paper.

\begin{theorem} \label{T:1} The following statements are true.
\begin{itemize}
\item [{\rm (a)}] $\dim (R/\frak a(M))<d$ if and only if $R/\frak p$ is unmixed and the formal fiber of  $R$ at $\frak p$ is Cohen-Macaulay for all $\frak p\in\Supp_R(M)$ with $\dim (R/\frak p)=d$.
\item[{\rm (b)}]  If $\frak p\in\Supp_R(M)$ with $\dim(R/\frak p)>\dim (R/\frak a(M))$ then  $R/\frak p$ is unmixed and the formal fiber of  $R$ at $\frak p$ is Cohen-Macaulay.
\end{itemize}
\end{theorem}

\begin{proof} (a) Suppose that $\dim (R/\frak a(M))<d.$ Let  $\frak p\in\Supp_R(M)$ with $\dim (R/\frak p)=d$. Then $\frak p\in\min\Ass_R(M)$. We need to prove that  $R/\frak p$ is unmixed and the generic formal fiber of  $R/\frak p$ is Cohen-Macaulay.

 Suppose in contrary that $R/\frak p$ is not unmixed. Then there exists $\frak P\in\Ass (\widehat R/\frak p\widehat R)$ such that $\dim (\widehat R/\frak P)<d.$ It follows  by \cite[Theorem 23.2]{Mat} that  $\frak p=\frak P\cap R$ and  $\frak P\in\Ass_{\widehat R}(\widehat M)$.  Set $k=\dim (\widehat R/\frak P).$  Note that $H^k_{\frak m}(M)\cong H^k_{\frak m\widehat R}(\widehat M).$ Therefore we get by Lemma \ref{L:1b}(a) that  $\frak P\in\Att_{\widehat R}H^k_{\frak m}(M)$.  Hence $\frak p\in\Att_RH^k_{\frak m}(M)$ by  Lemma \ref{L:1a}(b). It follows by Lemma \ref{L:1a}(a) that $d=\dim (R/\frak p)\leq\dim (R/\frak a_k(M)).$  Since $k<d,$ we have $\dim (R/\frak a(M))\geq d.$ This gives a contradiction.

 Let  $\frak P\in\Spec(\widehat R)$ such that $\frak p=\frak P\cap R.$  Suppose in contrary that  the ring $\widehat R_{\frak P}/\frak p\widehat R_{\frak P}$ is not Cohen-Macaulay. Since the natural map $R_{\frak p}\rightarrow \widehat R_{\frak P}$ is faithully flat and $\widehat M_{\frak P}\cong M_{\frak p}\otimes_{R_{\frak p}}\widehat R_{\frak P}$, it follows by Lemma \ref{L:bs} that
\begin{align}&\dim _{\widehat R_{\frak P}} (\widehat M_{\frak P})=\dim_{R_{\frak p}}(M_{\frak p})+\dim (\widehat R_{\frak P}/\frak p\widehat R_{\frak P})\notag\\
&\depth _{\widehat R_{\frak P}} (\widehat M_{\frak P})=\depth_{R_{\frak p}}(M_{\frak p})+\depth (\widehat R_{\frak P}/\frak p\widehat R_{\frak P}).\notag\end{align}
  Because $\widehat R_{\frak P}/\frak p\widehat R_{\frak P}$ is not Cohen-Macaulay, it follows by the above equations that $\widehat M_{\frak P}$ is not Cohen-Macaulay.  Set $\dim _{\widehat R_{\frak P}} (\widehat M_{\frak P})=t.$ Then  $H^i_{\frak P\widehat R_{\frak P}}(\widehat M_{\frak P})\neq 0$ for some $i<t.$ There exists by Lemma \ref{L:1a}(a) a prime ideal $\frak Q$ of $\widehat R$ such that $\frak Q\subseteq \frak P$ and  $\frak Q\widehat R_{\frak P}\in \Att_{\widehat R_{\frak P}}H^i_{\frak P\widehat R_{\frak P}}(\widehat M_{\frak P}).$ It follows by the shifted localization principle (see \cite[11.3.2]{BS}) that $\frak Q\in\Att_{\widehat R}H^{i+\dim (\widehat R/\frak P)}_{\frak m\widehat R}(\widehat M).$ It is clear that $\Ann_R(M)\subseteq \frak Q\cap R\subseteq \frak P\cap R=\frak p$. Since $\frak p\in\min\Ass_R(M)$, it follows that $\frak Q\cap R=\frak p.$  Hence  $\frak p\in\Att_RH^j_{\frak m}(M)$ by Lemma \ref{L:1a}(b), where  $j=i+\dim (\widehat R/\frak P).$ Hence $\frak p\supseteq \frak a_j(M)$ by Lemma \ref{L:1a}(a). As $\dim (R/\frak p)=d,$ it follows that  $\dim (R/\frak a_j(M))=d.$ Since $i<t,$ we have 
$$j< t+\dim (\widehat R/\frak P)=\dim_{\widehat R_{\frak P}}(\widehat M_{\frak P})+\dim (\widehat R/\frak P)\leq d.$$ It follows that $\dim (R/\frak a(M))=d.$ This gives a contradiction.

Conversely, suppose that $R/\frak p$ is unmixed and the generic formal fiber of  $R/\frak p$ is Cohen-Macaulay for all $\frak p\in\Supp_R(M)$ with $\dim (R/\frak p)=d$. We will show that $\dim (R/\frak a(M))<d.$ Suppose in contrary that $\dim (R/\frak a(M))=d.$ Then $\dim (R/\frak a_i(M))=d$ for some non-negative integer  $i<d$.  Let $\frak p\in\min\Var (\frak a_i(M))$ such that $\dim (R/\frak p)=d.$ Then $\frak p\in \Att_RH^i_{\frak m}(M)$ by  Lemma \ref{L:1a}(a). As $H^i_{\frak m}(M)\cong H^i_{\frak m\widehat R}(\widehat M),$ there exists by Lemma \ref{L:1a}(b) a prime ideal $\frak P\in\Att_{\widehat R}H^i_{\frak m\widehat R}(\widehat M)$ such that $\frak p=\frak P\cap R.$ Hence $\frak P\widehat R_{\frak P}\in \Att_{\widehat R_{\frak P}}H^{i-\dim (\widehat R/\frak P)}_{\frak P\widehat R_{\frak P}}(\widehat M_{\frak P})$ by the shifted localization principle, see \cite[Theorem 1.1]{NQ}.  It follows that $H^{i-\dim (\widehat R/\frak P)}_{\frak P\widehat R_{\frak P}}(\widehat M_{\frak P})\neq 0$ by Lemma \ref{L:1a}(a).  Because $\dim (R/\frak p)=d$, we get by our assumption  that $R/\frak p$ is unmixed. Let  $\frak Q\in\Ass (\widehat R/\frak p\widehat R)$ such that  $\frak Q\subseteq \frak P.$ Since $R/\frak p$ is unmixed, $\dim (\widehat R/\frak Q)=d.$ Note that $\widehat R$ is catenary, so we have 
$$\dim (\widehat R/\frak P)+\dim (\widehat R_{\frak P}/\frak p \widehat R_{\frak P})\geq \dim (\widehat R/\frak P)+\text{ht} (\frak P/\frak Q)=\dim (\widehat R/\frak Q)=d.$$ Hence $\dim (\widehat R/\frak P)+\dim (\widehat R_{\frak P}/\frak p \widehat R_{\frak P})=d.$ Therefore,
$$i-\dim (\widehat R/\frak P)<d-\dim (\widehat R/\frak P)=\dim (\widehat R_{\frak P}/\frak p \widehat R_{\frak P}).$$ Since the map $R_{\frak p}\rightarrow \widehat R_{\frak P}$ is faithfully flat and $\frak p\in\min\Ass_R(M)$, it follows  by Lemma \ref{L:bs} that
$$\dim_{\widehat R_{\frak P}}(\widehat M_{\frak P})=\dim_{R_{\frak p}} (M_{\frak p})+\dim (\widehat R_{\frak P}/\frak p \widehat R_{\frak P})=\dim (\widehat R_{\frak P}/\frak p \widehat R_{\frak P}).$$
 Hence $i-\dim (\widehat R/\frak P)<\dim_{\widehat R_{\frak P}}(\widehat M_{\frak P}).$ Since $H^{i-\dim (\widehat R/\frak P)}_{\frak P\widehat R_{\frak P}}(\widehat M_{\frak P})\neq 0$, it follows that  $\widehat M_{\frak P}$ is not Cohen-Macaulay.  As   $\frak p\in\min\Ass_R(M)$, it is clear that  $M_{\frak p}$ is Cohen-Macaulay of dimension $0$. So, it follows  by Lemma \ref{L:bs} for the flat local homomorphism $R_{\frak p}\to \widehat R_{\frak P}$ that $\widehat R_{\frak P}/\frak p\widehat R_{\frak P}$ is not Cohen-Macaulay.  This gives a contradiction.

(b) Set $\dim (R/\frak a(M)):=d-n$.  We use induction on $n$. The case $n=0$ is nothing to prove.  The case $n=1$ follows immediately by assertion (a). Now let $n\geq 2.$ Let $\frak p\in\Supp_R(M)$ with $\dim (R/\frak p)>d-n.$ If $\dim (R/\frak p)=d$ then $R/\frak p$ is unmixed and the formal fiber of  $R$ at $\frak p$ is Cohen-Macaulay by the assertion (a).  So, we assume that $\dim (R/\frak p)<d.$ As $\dim (R/\frak a(M))=d-n<d,$ we have $\dim\big(R/(\frak p\cap \frak a(M)^3)\big)<d$. Therefore, we can choose an element $x\in \frak p\cap \frak a(M)^3$ such that $x$ is a parameter of $M$. Denote by $U_M(0)$ the largest submodule of $M$ of dimension less than $d$. Set $\overline M:=M/U_M(0).$ Then we get by Lemma \ref{L:1d} the following isomorphisms
 $$H^i_{\frak m}(M/xM)\cong H^i_{\frak m}(M)\oplus H^{i+1}_{\frak m}(\overline{M})$$
for all $i\leq d-2.$ Set $k:=d-n.$ We first claim that $\dim U_M(0)\leq k$. In fact, suppose in contrary that $\dim U_M(0)> k$. Let $\frak q\in\Ass_R(U_M(0))$ such that $\dim (R/\frak q)=\dim_RU_M(0).$ Set $t:=\dim (R/\frak q).$ Then $k<t<d.$ Note that $\frak q\in\Ass_R(M).$ Therefore $\frak q\in\Att_R H^t_{\frak m}(M)$ by Lemma \ref{L:1b}(a). So we have by Lemma \ref{L:1a}(a) that  $\dim (R/\frak a_t(M))\geq \dim (R/\frak q)=t>k$. Since $t<d,$ we have $\dim (R/\frak a(M))\geq t>k.$ This gives a contradiction. So, the claim is proved. From the exact sequence
$0\to U_M(0)\to M\to \overline M\to 0$, we get the exact sequences
$$ H^i_{\frak m}(M)\to H^i_{\frak m}(\overline M)\to H^{i+1}_{\frak m}(U_M(0))$$
for all $i<d.$ Since $\dim U_M(0)\leq k$ by the claim and $\dim (R/\frak a_i(M))\leq k$, we have by the above exact sequences that  $\dim (R/\frak a_i(\overline M))\leq k$ for all $i<d.$ Therefore, it follows by the above isomorphisms that $\dim (R/\frak a_i(M/xM))\leq k$ for all $i\leq d-2.$ So, we have $\dim (R/\frak a(M/xM))\leq k.$ Since $\frak p\in\Supp_R(M)$ and $x\in\frak p$, it follows that $\frak p\in\Supp_R(M/xM).$ Note that $\dim (M/xM)=d-1$ and $\dim (R/\frak p)>k=(d-1)-(n-1).$ Therefore, we get by induction that $R/\frak p$ is unmixed and the formal fiber of  $R$ at $\frak p$ is Cohen-Macaulay.
\end{proof}

\begin{corollary} \label{C:1d} Denote by $U_R(0)$ the largest submodule of $R$ of dimenssion less than $\dim (R)$.  The following statements are equivalent:
	\begin{itemize}
		\item[\rm{(a)}] $\dim (R/\frak a(R))<\dim (R).$
		\item[\rm{(b)}]  $\dim (R/\frak a(M))<\dim (R)$ for any finitely generated $R$-module $M$ of dimension $\dim (R)$. 
	\end{itemize}
	If the conditions (a), (b) satisfy then  $R/U_R(0)$ is universally catenary and the formal fiber of $R$ at $\frak p$ is Cohen-Macaulay for any prime ideal $\frak p$ of $R$ with $\dim (R/\frak p)=\dim (R)$.
\end{corollary}

\begin{proof} The direction (b)$\Rightarrow$(a) is obvious.
	
	(a)$\Rightarrow$(b). Let $M$ be a finitely generated $R$-module of dimension $\dim (R)$. Assume that $\frak p\in\Supp_R(M)$ with $\dim (R/\frak p)=\dim (R)$. Then, we get by assumption (a) and Theorem \ref{T:1} that  $R/\frak p$ is  unmixed and the formal fiber of $R$ at $\frak p$ is Cohen-Macaulay. Therefore $\dim (R/\frak a(M))<\dim (R)$ by Theorem \ref{T:1}.  
	
	Note that $\Ass_R(R/U_R(0))=\{\frak p\in\Spec(R)\mid \dim (R/\frak p)=\dim (R)\}.$ Therefore, the rest statement follows immediately by \cite[Theorems 31.6, 31.7]{Mat}.   
\end{proof}

\begin{corollary} \label{C:2a} Suppose that $\dim (R/\frak a(M))\leq 1$. Then $R/\Ann_R(M)$ is a quotient of a Cohen-Macaulay local ring. In particular, $\nCM(M)=\{\frak P\cap R\mid \frak P\in\nCM(\widehat M)\}$, $\nCM(M)$ is  closed under Zariski topology and $\dim_R\nCM(M)=\dim_{\widehat R}\nCM(\widehat M).$
\end{corollary}
 
\begin{proof} Let $\frak p\in\Supp_R(M)$.  If $\dim (R/\frak p)\geq 2$ then we have by Theorem \ref{T:1} that $R/\frak p$ is unmixed and the formal fiber of  $R$ at $\frak p$ is Cohen-Macaulay. Assume that $\dim (R/\frak p)= 1$. Let $\frak P\in \Ass(\widehat{R}/\frak p\widehat{R}).$ Then $\frak P\cap R=\frak p.$ Hence $\frak P\neq \frak m\widehat{R}.$ Therefore $\dim(\widehat{R}/\frak P)=1.$ So, $R/\frak p$ is unmixed and the formal fiber of  $R$ at $\frak p$ is Cohen-Macaulay.  If $\dim(R/\frak p)=0$ then it is clear that  $R/\frak p$ is unmixed and the formal fiber of  $R$ at $\frak p$ is Cohen-Macaulay. Since $R/\frak p$ is unmixed for all $\frak p\in\Supp_R(M),$ it follows by \cite[Theorems 31.6, 31.7]{Mat} that $R/\Ann_R(M)$ is universally catenary. Therefore, it follows by Lemma \ref{L:1c} that $R/\Ann_R(M)$ is a quotient of a Cohen-Macaulay local ring. The rest statement follows by Lemma \ref{P:1e}.
\end{proof}

The following example shows that the assumption $\dim (R/\frak a(M))\leq 1$ in Corollary \ref{C:2a} is necessary.

\begin{example} \label{E:2a} {\rm Consider the Noetherian local domain $(R,\frak m)$ of dimension $2$ constructed by  Ferrand and  Raynaud \cite{FR} such that $R$ has an embedded prime $\frak Q$ with $\dim (\widehat R/\frak Q)=1$. Then $R$ is not a quotient of a Cohen-Macaulay local ring. Note that $\frak Q\in\Att_{\widehat R}H^1_{\frak m}(R)$ by Lemma \ref{L:1b}(a). As $\frak Q\in\Ass (\widehat R)$, we have $\frak Q\cap R\in\Ass (R)$. Since $R$ is a domain, $\frak Q\cap R=0.$ Hence $0\in\Att_RH^1_{\frak m}(R)$ by Lemma \ref{L:1a}(b). Therefore, $\Ann_RH^1_{\frak m}(R)=0$ by Lemma \ref{L:1a}(a). Therefore, $\frak a(R)=0$ and hence $\dim (R/\frak a(R))=2.$} 
\end{example}

\begin{corollary} \label{C:2b}  Suppose that $\dim (R/\frak a(M))=k$. Then the set $$\{\frak p\in \min\nCM(M)\mid \dim (R/\frak p)\geq k-1\}$$ is finite. In particular, if $\dim (R/\frak a(M))\leq 2$ then $\nCM(M)$ is  closed under Zariski topology.
\end{corollary}

\begin{proof} Set $X=\{\frak P\cap R\mid \frak P\in\min\nCM(\widehat M)\}.$ Then $X$ is a finite set. Let $\frak p\in\min\nCM(M)$ such that $\dim (R/\frak p)\geq k-1$. We will show that $\frak p\in X$. 
	
	Note that $\dim_R(\nCM(M))\leq k$ by \cite[Theorem 1.1]{NNK}. Therefore, $\dim (R/\frak p)=k$ or $\dim (R/\frak p)=k-1$. Let $\frak P\in\min\Var (\frak p\widehat R)$ such that $\dim (\widehat R/\frak P)=\dim (R/\frak p).$  Then $\frak p=\frak P\cap R.$ Since $\frak p\in\nCM(M),$ it follows by Lemma \ref{L:bs} that  $\frak P\in\nCM(\widehat M).$ Let $\frak Q\in\min \nCM(\widehat M)$ such that $\frak Q\subseteq \frak P.$ Set $\frak q:=\frak Q\cap R.$ It is clear that $\frak q\subseteq \frak p$ and $\frak q\in\Supp_R(M)$.  Consider the following two cases:

{\it Case 1: $\dim (R/\frak p)=k.$}  Then $\dim (\widehat R/\frak P)=k$ and $\dim (R/\frak q)\geq k.$ If $\frak Q\neq \frak P$ then $\dim (\widehat R/\frak Q)>k.$ Hence $\dim (R/\frak q)>k$ and hence $\frak p\neq \frak q.$  So, we get by Theorem \ref{T:1} that $\widehat R_{\frak Q}/\frak q\widehat R_{\frak Q}$ is Cohen-Macaulay. As $\frak Q\in \nCM(\widehat M)$, it follows by Lemma \ref{L:bs} that $\frak q\in\nCM(M)$. This is impossible since $\frak p\in\min\nCM(M)$. Therefore $\frak Q=\frak P$. Hence $\frak p\in X.$

{\it Case 2: $\dim (R/\frak p)=k-1.$} Then $\dim (\widehat R/\frak P)=k-1.$ Suppose that $\frak Q\neq \frak P$. Then $\dim (\widehat R/\frak Q)\geq k.$ Hence $\dim (R/\frak q)\geq k$ and hence $\frak p\neq \frak q.$ If $\dim (R/\frak q)=k$ then we have $\frak Q\in\min\Var(\frak q\widehat R), $ so $\widehat R_{\frak Q}/\frak q\widehat R_{\frak Q}$ is Cohen-Macaulay of dimension $0$. Since $\frak Q\in \nCM(\widehat M)$, we have $\frak q\in\nCM(M)$ by Lemma \ref{L:bs}. This is impossible by the minimality of $\frak p$ in $\nCM (M)$. Therefore, $\dim (R/\frak q)>k.$ Hence $\widehat R_{\frak Q}/\frak q\widehat R_{\frak Q}$ is Cohen-Macaulay by Theorem \ref{T:1}. Note that $\frak Q\in \nCM(\widehat M)$. Therefore we have by Lemma \ref{L:bs} that $\frak q\in\nCM(M)$. This is impossible because $\frak p\in\min\nCM (M)$. Hence $\frak Q=\frak P$, so $\frak p\in X.$ 

Thus, in any case we have 
$$\{\frak p\in\min\nCM(M)\mid \dim (R/\frak p)\geq k-1\}\subseteq X.$$
 Hence $\{\frak p\in\min\nCM(M)\mid \dim (R/\frak p)\geq k-1\}$ is a finite set. 

Now, assume that $\dim (R/\frak a(M))\leq 2.$  If $\dim (R/\frak a(M))\leq 1$ then the result follows by Corollary \ref{C:2a}. So we assume that $\dim (R/\frak a(M))=2.$ It follows by the above fact that  the set $\{\frak p\in\min\nCM(M)\mid \dim (R/\frak p)\geq 1\}$ is a finite set. Hence $\min\nCM(M)$ is a finite set. Therefore, $\nCM(M)$ is closed  under Zariski topology.  
\end{proof}

The following example shows that there exists a Noetherian local domain $(R,\frak m)$ such that $\dim (R/\frak a(R))=3$ and the $\nCM(R)$ is not closed under Zariski topology. Therefore, the assumption  $\dim (R/\frak a(M))\leq 2$ in Corollary \ref{C:2b} is necessary.

\begin{example}\label{E:2b} {\rm There exists by \cite[Example 3.1]{BS1} a Noetherian local domain $(R,\frak m)$ of dimension $3$ with the following two properties:
\begin{itemize}
	\item [{\rm (a)}] $\widehat R$ can be identified with $B/\frak I$, where
	$B:=\Bbb Q[[V_1, V_2, X, Y]]$ and $\frak I:= (V_1V_2)\cap (V_1^2,V_2^2)$,
	and $V_1, V_2, X, Y$ are independent indeterminates over $\Bbb Q$;
	\item [{\rm (b)}] $\nCM (R)=\{\frak p\in\Spec(R)\mid \depth (R_{\frak p})=\dim(R_{\frak p})-1\}$ which is not closed under Zariski topology.
\end{itemize}
Set $\frak P:=(V_1, V_2).$ As $\frak P\in\Ass (\widehat R)$ and $\dim (\widehat R/\frak P)=2,$ we get by Lemma \ref{L:1b}(a) that $\frak P\in\Att_{\widehat R}H^2_{\frak m}(R).$ Since $R$ is a domain and $\frak P\in\Ass (\widehat R)$, we have $\frak P\cap R=0.$ Hence $0\in \Att_RH^2_{\frak m}(R)$ by Lemma \ref{L:1a}(b). Therefore $\Ann_RH^2_{\frak m}(R)=0$ by Lemma \ref{L:1a}(a). Hence $\frak a(R)=0$ and hence $\dim (R/\frak a(R))=3.$}
\end{example}

Now we prove the second main result of this paper.

 \begin{theorem} \label{T:2}  Let $H^0_{\frak m}(M)=D_t\subset \ldots \subset D_1\subset D_0=M$ be the dimension filtration of $M$. For each $i\geq 1,$ set ${\cal D}_i:=D_{i-1}/D_i$. The following statements are true.
	\begin{itemize}
		\item [{\rm (a)}] If $\dim (R/\frak a({\cal D}_i))\leq 1$ for all $i$ then $R/\Ann_R(M)$ is a quotient of a Cohen-Macaulay local ring, $\nCM(M)$ is closed under Zariski topology and $\dim(\nCM(M))=\dim(\nCM(\widehat M)).$
		\item[{\rm (b)}]  If $\dim (R/\frak a({\cal D}_i))\leq 2$ for all $i$ then $R/\Ann_R(M)$ is catenary and $\nCM(M)$ is closed under Zariski topology.
	\end{itemize}
\end{theorem}  
\begin{proof} (a) Let $ \frak p\in\Supp_R(M).$ If $\frak p=\frak m$ then it is clear that $R/\frak p$ is unmixed and the formal fiber of $R$ at $\frak p$ is Cohen-Macaulay. If $\frak p\neq \frak m$ then $\frak p\in \Supp_R({\cal D}_i)$ for some $i.$ As  $\dim (R/\frak a({\cal D}_i))\leq 1$, it follows by Theorem \ref{T:1} that $R/\frak p$ is unmixed and the formal fiber of $R$ at $\frak p$ is Cohen-Macaulay.  Therefore $R/\Ann_R(M)$ is universally catenary and all of its formal fibers are Cohen-Macaulay. So, $R/\Ann_R(M)$ is a quotient of a Cohen-Macaulay local ring by Lemma \ref{L:1c}. The rest statement follows by Lemma \ref{P:1e}.
	
	(b)  Let $\frak p, \frak q\in\Supp_R(M)$  such that $\frak p\subset \frak q$. Then $\frak p\in\Supp_R({\cal D}_i)$ for some $i\leq t$. If $\dim (R/\frak p)\leq 1$ then it is clear that $R/\frak p$ is catenary. Assume that $\dim(R/\frak p)=2.$ Then the length of each saturated chain of prime ideals between $\frak p$ and $\frak m$ must be $2$. Therefore, $R/\frak p$ is catenary. Suppose that  $\dim (R/\frak p)>2$. Then $\dim (R/\frak p)>\dim (R/\frak a({\cal D}_i)).$ Hence $R/\frak p$ is unmixed by Theorem \ref{T:1}. Hence $R/\frak p$ is catenary. So, all saturated chains of prime ideals from $\frak p$ to $\frak q$ have the same finite length. Therefore, $R/\Ann_R(M)$ is catenary.
	
	In order to prove the closedness of $\nCM(M)$ under Zariski topology, it is enough to prove that $\min\nCM(M)$ is a finite set. Let $\frak p\in\min\nCM(M).$ If $\frak p=\frak m$ then $\nCM(M)=\{\frak m\}$ which is closed. So we can assume that $\frak p\neq \frak m$. Note that $(D_t)_{\frak p}=0$ and $M_{\frak p}$ is not Cohen-Macaulay. Therefore, we can choose  the least integer $i$ such that $1\leq i\leq t$ and $(D_i)_{\frak p}\neq M_{\frak p}.$  Then $(D_{i-1})_{\frak p}=M_{\frak p}$ and $(D_{i-1})_{\frak p}\neq (D_i)_{\frak p}$. So,  $\frak p\in\Supp_R(D_{i-1}/D_i).$
	
	We consider two cases.
	
	\noindent{\it Case 1: $(D_i)_{\frak p}=0$.} We claim that $\frak p\in\min\nCM(D_{i-1}/D_i).$ Suppose in contrary that $\frak p\notin\min\nCM(D_{i-1}/D_i).$ Since $(D_i)_{\frak p}=0$, we have
	$$M_{\frak p}=(D_{i-1})_{\frak p}=(D_{i-1}/D_i)_{\frak p}.$$ Since $\frak p\in\nCM(M),$ we have $\frak p\in\nCM(D_{i-1}/D_i).$  Since $\frak p\notin\min\nCM(D_{i-1}/D_i),$ there exists $\frak q\in\nCM (D_{i-1}/D_i)$ such that $\frak q\subset \frak p$ and $\frak q\neq \frak p$. As $\frak p\in\min\nCM(M),$ we have $\frak q\notin\nCM(M)$. Note that $(D_i)_{\frak q}=((D_i)_{\frak p})_{\frak q R_{\frak p}}=0$. So we get that
	$$M_{\frak q}=(M_{\frak p})_{\frak q R_{\frak p}}=((D_{i-1})_{\frak p})_{\frak q R_{\frak p}}=(D_{i-1})_{\frak q}=(D_{i-1}/D_i)_{\frak q}.$$ Since $\frak q\in\nCM (D_{i-1}/D_i)$, we have $\frak q\in\nCM(M).$  This gives a contradiction.
	
	\noindent{\it Case 2: $(D_i)_{\frak p}\neq 0$.} We claim that $\frak p\in\min\Supp_R(D_{i-1}/D_i).$	 Suppose in contrary that $\frak p\notin\min\Supp_R(D_{i-1}/D_i).$ Let $\frak q\in\Supp_R(D_{i-1}/D_i)$ such that $\frak q\subset \frak p$ and $\frak q\neq \frak p$. Since $\frak q\in\Supp_R(D_{i-1}/D_i)$, we have $(D_{i-1})_{\frak q}\neq (D_i)_{\frak q}.$ Since $(D_{i-1})_{\frak p}=M_{\frak p}$ and $\frak q\subset \frak p$, we have $(D_{i-1})_{\frak q}=M_{\frak q}.$ So, it follows by the catenarity of $R/\Ann_R(M)$ that 
$$\dim_{R_{\frak q}}((D_i)_{\frak q})<\dim_{R_{\frak q}}((D_{i-1})_{\frak q})=\dim_{R_{\frak q}}(M_{\frak q}).$$
	 Note that $M_{\frak q}$ is Cohen-Macaulay since $\frak p\in\min\nCM(M).$ Hence $(D_i)_{\frak q}=0.$ This gives a contradiction. 
	 
	 From the above two cases, we have $$\min\nCM(M)\subseteq \bigcup_{i=1}^t\big(\min\nCM({\cal D}_i)\cup \min \Supp_R({\cal D}_i)\big).$$
	 Therefore $\min\nCM(M)$ is a finite set by Corollary \ref{C:2b}. Thus, $\nCM(M)$ is closed under Zariski topology.
	\end{proof}

		Finally, we give examples to illustrate the main results. 	Consider the case of non-unmixed Noetherian local domains.
		
	\begin{example} {\rm Let $d\geq 2$ be an integer. Let $(R,\frak m)$ be a non-unmixed Noetherian local domain of dimension $d$ (such a local domain exists, see \cite[Example 3.8]{ChNN}). Then there exists $\frak P\in\Ass (\widehat R)$ such that $\dim (\widehat R/\frak P)<d.$ Since $R$ is a domain and $\frak P\cap R\in\Ass(R)$, it follows that $\frak P\cap R=0.$ Hence, the generic formal fiber of $R$ is not Cohen-Macaulay. Therefore, $\dim (R/\frak a(R))=d$ by Theorem \ref{T:1}.}	
	\end{example}

	Consider the case of unmixed Noetherian local domains.
	\begin{example} \label{E:2} {\rm Let $d\geq 3$ be an integer and let $K$ be a field of characteristic $0$. We denote by  $S=K[[x_1, \ldots x_{2d-1}]]$  the ring of formal power series in $2d-1$ variables over $K$. Set $T=S/I\cap J$, where $I=(x_1, x_2,\ldots , x_{d-1})S$ and $J=(x_d, x_{d+1}, \ldots , x_{2d-2})S.$ Consider the two subsets  $W_1:=\Ass (T)$ and $W_2:=\{\frak A\in\Spec(T)\mid \frak A\subseteq (I+J)T\}$ of $\Spec(T)$.  Then there exist by \cite[Lemma 2.8, Theorem 3.1]{CL}  Noetherian local domains $(R_1,\frak m_1)$ and $(R_2,\frak m_2)$ such that for  each $i\in\{1, 2\}$, the following three properties satisfy:  $\widehat R_i\cong T$;  $W_i$ is the generic formal fiber of $R_i$; and for each non-zero prime ideal $\frak p$ of $R_i$, there exists uniquely a prime ideal $\frak P$ of $\widehat R_i$ such that $\frak P\cap R_i=\frak p$. Set $\frak P_1=(I+J)\widehat R_1$, $\frak P_2=(I+J)\widehat R_2,$ $\frak p_1=\frak P_1\cap R_1$ and $\frak p_2=\frak P_2\cap R_2.$ Then $\dim (R_i)=d$ for $i\in\{1,2\}$ and}
	\begin{itemize}
	\item[{\rm(a)}] {\rm $\dim(R_1/\frak a(R_1))=1$,  $\nCM(\widehat R_1)=\{\frak P_1, \frak m_1\widehat R_1\},$ $\nCM(R_1)=\{\frak p_1, \frak m_1\}$, $R_1$ is a quotient of a Cohen-Macaulay local ring, and $\dim \nCM(\widehat R_1)=\dim \nCM(R_1)=1.$}
	\item[{\rm (b)}] {\rm $\dim(R_2/\frak a(R_2))=d$,  $\nCM(\widehat R_2)=\{\frak P_2, \frak m_2\widehat R_2\},$ $\nCM(R_2)=\{\frak m_2\}$, the generic formal fiber of $R_2$ is not Cohen-Macaulay, $R_2$ is not a quotient of a Cohen-Macaulay local ring, $\dim \nCM(\widehat R_2)=1$ and $\dim \nCM(R_2)=0.$}
	\end{itemize}
	\end{example}
	\begin{proof} (a) From the exact sequence $0\to \widehat R_1\to \widehat R_1/I\widehat R_1\oplus \widehat R_1/J\widehat R_1\to \widehat R_1/\frak P_1\to 0$ with notice that $\widehat R_1/I\widehat R_1$, $\widehat R_1/J\widehat R_1$ are Cohen-Macaulay of dimension $d$ and $\widehat R_1/\frak P_1$ is Cohen-Macaulay of dimension $1$, we have $H^i_{\frak m_1\widehat R_1} (\widehat R_1)=0$ for $i\notin\{2, d\}$ and $H^2_{\frak m_1\widehat R_1} (\widehat R_1)\cong H^1_{\frak m_1\widehat R_1} (\widehat R_1/\frak P_1)$. Therefore, $\Att_{\widehat R_1}H^2_{\frak m_1\widehat R_1} (\widehat R_1)=\{\frak P_1\}$ by \cite[Theorem 7.3.2]{BS} and $H^i_{\frak m_1}(R_1)=0$ for $i\notin\{2, d\}$.  Hence $\Att_{R_1}H^2_{\frak m_1}(R_1)=\{\frak p_1\}$ by Lemma \ref{L:1a}(b). Since $\frak P_1\notin W_1,$ we have $\frak p_1\neq 0$ and hence $\dim (R_1/\frak p_1)=1.$ Therefore, 				
		 $\dim (\widehat R_1/\frak a(\widehat R_1))=1$ and $\dim (R_1/\frak a(R_1))=1.$  Hence $R_1$ is a quotient of a Cohen-Macaulay local ring by Theorem \ref{T:2}(a). Moreover, since $R_1$ and $\widehat R_1$ are equidimensional, we get by \cite[Corollary 4.2(v)]{CNN} that $\nCM(\widehat R_1)=\Var(\frak a(\widehat R_1))=\{\frak P_1, \frak m_1\widehat R_1\}$ and $\nCM(R_1)=\Var(\frak p_1)=\{\frak p_1, \frak m_1\}.$ Therefore, $\dim \nCM(\widehat R_1)=\dim \nCM(R_1)=1.$
		 
		 (b) With the same argurments as in (a), we have $H^i_{\frak m_2\widehat R_2} (\widehat R_2)=0$ for $i\notin\{2, d\}$  and $H^2_{\frak m_2\widehat R_2} (\widehat R_2)\cong H^1_{\frak m_2\widehat R_2} (\widehat R_2/\frak P_2)$. Therefore, $\Att_{\widehat R_2}H^2_{\frak m_2\widehat R_2} (\widehat R_2)=\{\frak P_2\}$ and $H^i_{\frak m_2}(R_2)=0$ for $i\notin\{2, d\}$.  Hence $\Att_{R_2}H^2_{\frak m_2}(R_2)=\{\frak p_2\}$. Since $\frak P_2\in W_2,$ we have $\frak p_2=0$ and hence $\dim (R_2/\frak p_2)=d.$ Therefore, 				
		 $\dim (\widehat R_2/\frak a(\widehat R_2))=1$ and $\dim (R_2/\frak a(R_2))=d.$  Hence, the generic formal fiber of $R_2$ is not Cohen-Macaulay by Theorem \ref{T:1}(a), so $R_2$ is not a quotient of a Cohen-Macaulay local ring. Since  $\widehat R_2$ is equidimensional, we get by \cite[Corollary 4.2(v)]{CNN} that $\nCM(\widehat R_2)=\Var(\frak a(\widehat R_2))=\{\frak P_2, \frak m_2\widehat R_2\}$. Therefore,  $\nCM(R_2)\subseteq \{\frak p_2, \frak m_2\}$  by Lemma \ref{L:1e}. Note that $\frak p_2\notin\nCM(R_2)$ since  $\frak p_2=0.$  So, $\nCM(R_2)=\{\frak m_2\}$. Therefore, $\dim \nCM(\widehat R_2)=1$ and $\dim \nCM(R_2)=0.$
		\end{proof}

\end{document}